\documentclass[12pt]{amsart}
\usepackage{cmap}
\usepackage[T2A]{fontenc}
\usepackage[utf8]{inputenc}
\usepackage[ngerman, english]{babel}
\usepackage[pdftex,unicode]{hyperref}
\usepackage{amsmath}
\usepackage{amssymb}
\usepackage{amsthm}
\usepackage{verbatim}
\usepackage{amsfonts}
\usepackage{graphicx}
\usepackage[normalem]{ulem}
\usepackage{extsizes}
\usepackage{float}
\usepackage{bbold}
\usepackage{dsfont}
\usepackage{calc}
\usepackage{bm}
\newlength{\widecommentlength}
\newcommand{\widecommentbox}[3]{\def#1##1{\strut\newline\noindent\colorbox{#3}{\linespread{1}\parbox{.95\textwidth}{\small {\bf [#2]} ##1}}\newline}}

\widecommentbox{\alex}{Alex}{green!20!white}
\widecommentbox{\ad}{AD}{red!20!white}

\usepackage{enumitem}


\usepackage[square,numbers,sort&compress]{natbib}
\usepackage{mathtools}
\newcommand{\cal}[1]{\mathcal{#1}}
\renewcommand{\leq}{\leqslant}
\renewcommand{\geq}{\geqslant}
\renewcommand{\phi}{\varphi}
\newtheorem{theorem}{Theorem}
\newtheorem{prop}{Proposition}
\newtheorem*{prop*}{Proposition}
\newtheorem{lemma}{Lemma}
\theoremstyle{definition}
\newtheorem{corol}{Corollary}
\newtheorem*{corol*}{Corollary}

\newtheorem*{remark*}{Remark}

\newtheorem{definition}{Definition}
\newtheorem{example}{Example}

\newcommand\inner[2]{\langle #1, #2 \rangle}
\newcommand{\seq}[1]{\{{#1}_n\}_{n=0}^\infty}
\newcommand{\seqone}[1]{\{{#1}_n\}_{n=1}^\infty}
\newcommand{\fsys}{\mathfrak{F}}
\newcommand{\fstarsys}{\mathfrak{F^{*}}}

\numberwithin{remark}{section}
\numberwithin{theorem}{section}
\numberwithin{prop}{section}
\numberwithin{corol}{section}
\numberwithin{equation}{section}
\numberwithin{lemma}{section}

\newtheoremstyle{case}{}{}{}{}{}{:}{ }{}
\theoremstyle{case}

\begin{document}
\title{On $k$ point density problem for band-diagonal $M$-bases}
\author{Alexey Pyshkin}
\begin{abstract}
  In the early 1990s the works of Larson, Wogen and Argyros, Lambrou, Longstaff
    disclosed an example of a strong tridiagonal $M$-basis that was not rank one dense.
  Later Katavolos, Lambrou and Papadakis studied $k$ point density property of this example.
  In this paper we present new methods for the analysis of $k$ point density
    and rank one density properties for band-diagonal $M$-bases.
\end{abstract}
\thanks{The author was supported by RFBR (the project~16-01-00674) and by «Native towns», a social investment program of PJSC «Gazprom Neft».}
\maketitle

\section{Introduction}
  \subsection{Density properties}
    Consider the infinite-dimensional real Hilbert space $\cal{H}$.
    Suppose that $\cal{H}$ has an orthonormal basis $\{e_j\}_{j=0}^\infty$.
    The sequence $\fsys=\seq{f}$ is called \emph{minimal} if none of its elements can be approximated by the linear combinations of the others.
    The system $\fsys$ is complete and minimal when and only when it possesses a unique biorthogonal system $\fstarsys$.
    We call the minimal system $\fsys$ \emph{band-diagonal} if there exists $L \in \mathbb{N}$ such that $\inner{f_t}{e_l} = \inner{f^*_t}{e_l} = 0$
      whenever $\lvert t - l \rvert > L$.
    We say that $\fsys$ is an $M$-basis if $\fstarsys$ is complete as well.
    
    Consider the operator algebra $\cal{A} = \{T\in B(\cal{H}): Tf_n = \lambda_n f_n, \text{ for some } \lambda_n \in \mathbb{R}, n \geq 0\}$
      and the algebra $R_1(\cal{A})$ generated by rank one operators of $\cal{A}$.
    We are interested in the following properties of the algebra $\cal{A}$.
    \begin{definition}[$k$ point density property]
      \label{kpd}
      We say that the algebra $\cal{A}$ is \emph{$k$ point dense} if for any $x_1, x_2,\dots x_k \in \cal{H}$ and $\varepsilon > 0$
        there exists such $T\in \cal{A}$ that $||Tx_s - x_s|| < \varepsilon$ for any $1 \leq s \leq k$.
    \end{definition}
    The definition for $k=1$ is equivalent to $\fsys$ being a \emph{strong $M$-basis} (see~\cite{katavolos}):
      the system $\fsys$ is called a \emph{strong $M$-basis} if for any $x\in\cal{H}$ we have $x \in \overline{span}\big(\inner{x}{f^*_n}f_n\big)$, where
      $\overline{span}$ denotes the closed linear span.
    \begin{definition}[rank one density property]
      \label{r1d}
      We say that the algebra $\cal{A}$ is \emph{rank one dense} if the unit ball of rank one subalgebra $R_1(\cal{A})$
        is dense in the unit ball of $\cal{A}$ in the strong operator topology.
    \end{definition}
    By abuse of notation, we say that $\fsys$ is $k$ point dense (rank one dense)
      when the corresponding algebra $\cal{A}$ is $k$ point dense (rank one dense).

    Notice that rank one density property implies $k$ point density property for any $k$.

  \subsection{Motivation}


    Long\-staff in~\cite{longstaff} studied abstract subspace lattices and corresponding operator algebras.
    In that paper Longstaff raised an important question: does one point density property always imply rank one density property?

    The solution remained unknown until Larson and Wogen showed~\cite{larson} that the answer is negative.
    They constructed an example of a vector system $\fsys$ such that it is one point dense but does not possess rank one density property.
    \begin{example}[Larson--Wogen system $\fsys_{LW}$ parameterized with real $a_n$]
      \label{lw-sys}
      For any $j \geq 0$ we define
      \begin{align*}
        &f_{2j+1}=-a_{2j+1}e_{2j} + e_{2j+1} + a_{2j+2}e_{2j+2} \qquad &f_{2j}=e_{2j},\\
        &f^*_{2j}=-a_{2j}e_{2j-1} + e_{2j} + a_{2j+1}e_{2j+1} \qquad &f^*_{2j+1}=e_{2j+1},
      \end{align*}
      where $a_n$ are nonzero real numbers for any $n > 0$ and $a_0 = 0$.
    \end{example}
    The construction presented by Larson and Wogen was remarkably simple and elementary,~--- notice that the matrices corresponding to
      the vectors $\{f_j\}_{j=0}^\infty$ and $\{f^*_j\}_{j=0}^\infty$ are both tridiagonal.
    Afterwards this example was also studied in~\cite{argyroslambrou} (see Addendum), by Azoff and Shehada in~\cite{azoff}, in~\cite{me1}.
    In 1993 Katavolos, Lambrou and Papadakis in~\cite{katavolos} performed a deep analysis of the density properties
      of this vector system and deduced that for $\fsys_{LW}$ one point density does not imply rank one density.
    Moreover, they showed that for such system rank one density is equivalent to two point density.

    We are going to consider band-diagonal systems similar to the one regarded by Larson and Wogen and to determine the exact conditions
      for $k$ point density property of such vector systems.
    In this paper we present a few new techniques for analysis of $k$ point density and rank one density of band-diagonal vector systems.

    In the next section we will gather some basic facts and outline the main idea of the paper.
    In Section~\ref{section:lw-sys} we perform the analysis for Larson--Wogen example, providing the simpler proof of Theorems 2.1 and 2.2 in~\cite{katavolos}.
    In Section~\ref{section:pentadiagonal} we prove a similar theorem for one pentadiagonal system.
\section{Preliminaries}
  \label{sec:preliminaries}
  Suppose that $\fsys=\{f_n\}_{n=0}^\infty$ is an arbitrary band-diagonal $M$-basis and $\fstarsys = \{f^*_n\}_{n=0}^\infty$ is its biorthogonal sequence.
  In this section we establish several facts about $\fsys$.
  \begin{prop}
    The system $\fsys$ is rank one dense if and only if any trace class operator $T$,
      such that $\inner{Tf_n}{f_n^*} = 0$ for any $n \geq 0$, has  zero trace.
  \end{prop}
  \begin{proof}
    It is well known that rank one density property is equivalent to $R_1(\cal{A})$ being dense
      in $\cal{A}$ in the ultraweak (or $\sigma$-weak) topology (see~\cite{katavolos}, Theorem 2.2).
  \end{proof}
  \begin{prop}
    The system $\fsys$ is $k$ point dense if and only if any $k$-dimensional operator $T$,
      such that $\inner{Tf_n}{f_n^*} = 0$ for any $n \geq 0$, has  zero trace.
  \end{prop}
  \begin{proof}
    In the paper~\cite{katavolos} authors proved the proposition for $k = 2$.
    For greater $k$-s the same reasoning works.
  \end{proof}
  For an arbitrary linear operator $T$ we will be interested in the differences between
    the partial sums of the Fourier series using the system $\fsys$ and
    partial sums of the canonical Fourier series (using the orthonormal basis $\{e_n\}_{n=0}^\infty$):
  \begin{equation}
    \label{eq:xi}
    \Xi_n = \sum_{m=0}^n \inner{Tf_m}{f_m^*} - \sum_{m=0}^n \inner{Te_m}{e_m},
  \end{equation}
    where $\langle \cdot, \cdot\rangle$ denotes the scalar product in $\cal{H}$.
  It appears that $\Xi_n$ takes a concise and compact form for the finite-band system $\fsys$, and it is much easier to study $\Xi_n$
    than, for example, $\inner{Tf_m}{f_m^*}$.

  \begin{prop}
    \label{prop:reformulation}
    The operator $T$ is a trace class operator annihilating the subalgebra $R_1(\cal{A})$ if and only if
      for any $n \geq 0$ one has 
        \begin{equation}
          \label{eq:prop:ref}
          \Xi_n + \sum_{m=0}^n \inner{Te_m}{e_m} = 0.
        \end{equation}
  \end{prop}
    We will use this formulation in the following sections.

  \medskip
  Now consider the operator $T$ which has a finite rank.
  In that case we write $T$ as a finite sum $T = \sum_{s=1}^k y^s \otimes x^s$,
    where $x^s, y^s \in \cal{H}$.

  Let us define vectors $v_n$ and $u_n$ in $\mathbb{R}^k$ as follows:
  \begin{equation*}
    v_n = (x^1_n, x^2_n, \dots, x^k_n)\qquad
    u_n = (y^1_n, y^2_n, \dots, y^k_n),
  \end{equation*}
  where $x^s_n = \inner{x^s}{e_n}$ and $y^s_n = \inner{y^s}{e_n}$.
  Since $\inner{Te_m}{e_l} = \inner{u_m}{v_l}$ for any $m$ and $l$, we can rewrite $\Xi_n$ in terms of
    the scalar products of $\{u_n\}_{n=0}^\infty$ and $\{v_n\}_{n=0}^\infty$.
  In turn it means that~\eqref{eq:prop:ref} can be rewritten in terms of the scalar products
    of $\{u_n\}_{n=0}^\infty$ and $\{v_n\}_{n=0}^\infty$.

  Hence, the existence of $T$ might be reduced to the existence of
    the vectors $u_n$, $v_n$ in $\mathbb{R}^k$ such that the sequences $\{\lvert u_n\rvert\}_{n=0}^\infty$,
    $\{\lvert v_n\rvert\}_{n=0}^\infty$ are both square summable and~\eqref{eq:prop:ref} is satisfied.
  Thus, instead of looking for $k$ vectors $x^s$ and $y^s$ in $\cal{H}$, we might look for an infinite sequence
    of $k$-dimensional vectors $v_n$ and $u_n$ such that $\{v_n \}_{n=0}^\infty$, $\{u_n\}_{n=0}^\infty$ belong to $\ell^2(\mathbb{R}^k)$.
  That is one of the key ideas in our method of analysing $k$ point density property for $\fsys$.

  Thus, we have just found the following reformulation for $k$ point density property.
  \begin{prop}
    \label{prop:kreformulation}
    The following two statements are equivalent:
    \begin{enumerate}
      \item there exists a $k$-dimensional operator $T$ which annihilates $R_1(\cal{A})$
        such that $Tr\, T \neq 0$,
      \item there exist vectors $\{u_n\}_{n=0}^\infty$, $\{v_n\}_{n=0}^\infty$ in $\ell^2(\mathbb{R}^k)$ such that
        for the operator $T = \displaystyle\sum_{t,l=0}^\infty \inner{u_t}{v_l} e_t \otimes e_l$ we have
        \begin{equation}
          \label{eq:thm}
          \Xi_n + \sum_{m=0}^n \inner{Te_m}{e_m} = 0
        \end{equation}
          for any $n \geq 0$.
    \end{enumerate}
  \end{prop}
  As we already mentioned, the equation~\eqref{eq:thm} can be expressed via $u_n$ and $v_n$.
  Moreover, we can also write the trace of $T$ in terms of $u_n$, $v_n$:
    \begin{equation}
      \label{eq:propstar}
        Tr\,T = \sum_{s=1}^k \langle y^s, x^s \rangle = \sum_{s=1}^k \sum_{n=0}^\infty y^s_n x^s_n
              = \sum_{n=0}^\infty \sum_{s=1}^k y^s_n x^s_n = \sum_{n=0}^\infty \langle u_n, v_n \rangle.
    \end{equation}
  Essentially, $k$ point density property can be viewed as a
    possibility of laying out the sequence of vectors in $\mathbb{R}^k$ which are constrained with
    a series of relations~\eqref{eq:thm} and~\eqref{eq:propstar}.
\section{Classification for the Larson--Wogen $M$-basis}
  \label{section:lw-sys}
  In this section we study Larson--Wogen vector system $\fsys_{LW}$ (Example~\ref{lw-sys}).
  Namely, we prove a theorem similar to Theorem 2.2 of~\cite{katavolos}.
  Up until now there existed two different techniques in studying $k$ point density for $k=1$ (strong $M$-bases) and for $k\geq2$.
  Here we demonstrate a universal method for the analysis of $k$ point density property.
  \begin{theorem}[\cite{katavolos}, Theorem 2.2]
    \label{thm:katavolos}
    The sequences $\fsys_{LW}$ and $\fsys_{LW}^*$ are biorthogonal and both are complete in $\cal{H}$.
    Moreover, the following is true.
    \begin{enumerate}
      \item  the system $\fsys_{LW}$ is one point dense (a strong $M$-basis) if and only if the sequence
        \begin{equation}
          \mu_n = \frac{a_{n-1} a_{n-3} \dots}{a_{n} a_{n-2} \dots }
        \end{equation}
        does not belong to $\ell^2$.
      \item the system $\fsys_{LW}$ is $k$ point dense ($k > 1$) if and only if the sequence $\{1/a_n\}_{n=1}^\infty$ does not belong to $\ell^1$.
    \end{enumerate}
  \end{theorem}
  \begin{proof}
    Due to Proposition~\ref{prop:kreformulation} we know that $k$ point density of the system $\fsys$ is equivalent to the existence of $k$-dimensional
      vectors $u_n$, $v_n$ such that~\eqref{eq:thm} holds for the corresponding operator $T$.
    For the given $M$-basis $\fsys = \fsys_{LW}$ we can calculate $\Xi_n$ precisely:
    \begin{align*}
      \Xi_{2n-1} &= a_{2n}T_{2n - 1, 2n},\\
      \Xi_{2n} &= a_{2n + 1}T_{2n + 1, 2n},
    \end{align*}
      where $T_{ij} = \inner{Te_j}{e_i}$.

    Since $T_{ij} = \inner{u_j}{v_i}$, we have
    \begin{align*}
      \Xi_{2n-1} &= a_{2n}   \inner{u_{2n}}{v_{2n-1}},\\
      \Xi_{2n}   &= a_{2n+1} \inner{u_{2n}}{v_{2n+1}},
    \end{align*}
      where $\langle\cdot, \cdot\rangle$ denotes the scalar product in $\mathbb{R}^k$.

    For the convenience of the reader we will introduce the sequences of vectors $w_n$ and $w^*_n$.
    \begin{align*}
      w_{2n} &= u_{2n}, \quad w^*_{2n} = v_{2n},\\
      w_{2n+1} &= v_{2n+1}, \quad w^*_{2n+1} = u_{2n+1}.
    \end{align*}
    In view of this notation $\Xi_n = a_{n+1} \inner{w_n}{w_{n+1}}$ and due to Equation~\eqref{eq:propstar}
      $Tr\, T = \sum_{m=0}^\infty \inner{w_m}{w^*_m}$.

    Thus, we get that $\fsys$ is not $k$ point dense if and only if there exist
      $k$-dimensional vectors $\{w_n\}_{n=0}^\infty$, $\{w^*_n\}_{n=0}^\infty$ lying in $\ell^2(\mathbb{R}^k)$ such that
    \begin{equation}
      \label{eq:vector2}
      a_{n+1} \inner{w_n}{w_{n+1}} = -\sum_{m=0}^n \inner{w_m}{w^*_m},
    \end{equation}
      for any $n \geq 0$, and $\sum_{m=0}^\infty \inner{w_m}{w^*_m} \neq 0$.

    \medskip
    In what follows we show that the latter can be simplified even more.
    \begin{prop}
      \label{prop:reformulation-lw}
      The system $\fsys$ is not $k$ point dense if and only if there exists a sequence of vectors
        $\{r_n\}_{n=0}^\infty$ in $\ell^2(\mathbb{R}^k)$ such that
      \begin{equation}
        \label{eq:vector3}
        a_{n+1} \inner{r_n}{r_{n+1}} = 1,
      \end{equation}
        for any $n \geq 0$.
    \end{prop}
    \begin{proof}
      Suppose we found such $r_n$.
      Then we solve~\eqref{eq:vector2} by putting $w^*_n$ to zero, $w_n$ to $r_n$ for any $n > 0$ and
        choosing the vector $w^*_0$ so that $\inner{w_0}{w^*_0} = -1$.

      Now we prove the converse.
      Suppose we found such $w_n$ that~\eqref{eq:vector2} holds.
      Given that the vectors $w_n$ lie in $\mathbb{R}^k$, we rewrite the scalar product as
        the product of the vector lengths and the cosine of the angle between the vectors.
      Namely, we define $W_n = \lvert w_n\rvert$ and real $\theta_n$ that
        $\inner{w_{n}}{w_{n+1}} = W_n W_{n+1} \cos{\theta_n}.$

      The sequence $\Xi_n = -\sum_0^n \inner{w_m}{w^*_m}$ has a non-zero limit, so let us
        find the largest $N > 0$ such that $\Xi_N = 0$.
      Then we can modify the original sequence by setting $w_n$, $w^*_n$ to zero for any $0 \leq n \leq N$ so that~\eqref{eq:vector2}
        still holds.
      Therefore, without loss of generality we can assume that $\Xi_n \neq 0$ for any $n \geq 0$.
      Setting $a'_n = a_n \cos{\theta_n}$ we see that the sequence
      \[
        W_n = \frac{\Xi_{n-1}/a'_n}{\Xi_{n-2}/a'_{n-1}} \cdot \frac{\Xi_{n-3}/a'_{n-2}}{\Xi_{n-4}/a'_{n-3}} \cdots
      \]
        belongs to $\ell^2$.
      Now since $\Xi_n = -\sum_0^n \inner{w_m}{w^*_m}$, we discover that
      \[
        \frac{\Xi_n}{\Xi_{n-1}} = 1 + \eta_n,
      \]
        where $\{\eta_n\}_{n=1}^\infty \in \ell^1$.
      Thus the product of such $(1 + \eta_m)$ fractions is bounded by some constant above.
      It follows that the sequence
      \[
        W^\#_n = \frac{1/a'_n}{1/a'_{n-1}} \cdot \frac{1/a'_{n-2}}{1/a'_{n-3}} \cdots
      \]
        belongs to $\ell^2$.
      Now we set $r_n$ to $\frac{W^\#_n}{W_n}w_n$, and then~\eqref{eq:vector2} holds since
      \[
        a_{n+1} \inner{r_n}{r_{n+1}} = a_{n+1} \frac{1/a'_{n+1}}{\Xi_n/a'_{n+1}} \inner{w_n}{w_{n+1}} = 1.
      \]
      Since $\lvert r_n \rvert = \lvert W^\#_n \rvert$ and the sequence $\left\{\lvert W^\#_n \rvert\right\}_{n=1}^\infty$ belongs to $\ell^2$,
        the sequence $\{r_n\}_{n=1}^\infty$ belongs to $\ell^2(\mathbb{R}^k)$ as well.
    \end{proof}

    Now we are ready to prove the theorem for the case $k=1$.
    \begin{prop}
      The system $\fsys_{LW}$ is one point dense if and only if $\seqone{\mu}$ does not belong to $\ell^2$.
    \end{prop}
    \begin{proof}
      It follows from Proposition~\ref{prop:reformulation-lw}.

      The case $k=1$ has all the vectors $r_n$, $r^*_n$ lying on the same line ($\mathbb{R}^1$).
      Since all $r_n$ are collinear, the lengths of the vectors $r_n$ are precisely $\mu_n$.
      Hence, Equation~\eqref{eq:vector3} can be satisfied if and only if $\seqone{\mu}$ is square summable.
    \end{proof}

    After this we consider the case $k > 1$.
    \begin{prop}
      The system $\fsys_{LW}$ is $k$ point dense \textup($k > 1$\textup) if and only if the sequence $\{1/a_n\}_{n=1}^\infty$
        does not belong to $\ell^1$.
    \end{prop}
    \begin{proof}
      According to Proposition~\ref{prop:reformulation-lw}, the system $\fsys_{LW}$ is $k$ point dense
        if and only if there is no such sequence $\seq{r}$ in $\ell^2(\mathbb{R}^k)$ which satisfy $a_n \inner{r_n}{r_{n-1}} = 1$.
      Obviously, if there are such vectors $r_n$, then $\{1/a_n\}_{n=1}^\infty$ belongs to $\ell^1$.

      Conversely, suppose $\{1/a_n\}_{n=1}^\infty$ belongs to $\ell^1$.
      Then the sequence $R_n = \max(\lvert a_n \rvert^{-\frac{1}{2}}, \lvert a_{n+1} \rvert^{-\frac{1}{2}})$ is square summable.
      Observe that $R_nR_{n-1} \geq 1/\lvert a_n\rvert$, and so it is always possible to choose the angle $\theta_n$ so that
      \[
        a_n \langle r_n, r_{n-1} \rangle = a_n R_n R_{n-1}\cos{\theta_n} = 1.
      \]
      Now we have defined the lengths for $r_n$ and the angles between each two consecutive vectors $r_{n-1}$, $r_n$.
      Obviously, for any $k \geq 2$ we are able to lay out the vectors $r_n$ in $\mathbb{R}^k$.
    \end{proof}
    The last two propositions prove Theorem~\ref{thm:katavolos}.
  \end{proof}

\section{Pentadiagonal example}
  \label{section:pentadiagonal}
  In this section we explore another vector system $\fsys$ and its biorthogonal system $\fstarsys$ defined as follows:
  \begin{equation*}
    \label{eq:5system}
    \begin{aligned}
      &\mathbf{f_{4j}} = e_{4j}, \quad
      \mathbf{f^*_{4j}} = e_{4j} + d_{2j - 1} e_{4j-2} - b_{2j-1} e_{4j-1} + a_{2j} e_{4j+1} + c_{2j} e_{4j+2}\\
      &\mathbf{f_{4j+1}} = -a_{2j} e_{4j} + e_{4j+1}, \quad
      \mathbf{f^*_{4j+1}} = e_{4j+1} + b_{2j} e_{4j+2},\\
      &\mathbf{f_{4j+2}} = e_{4j+2} + d_{2j} e_{4j} - b_{2j} e_{4j+1} + a_{2j+1} e_{4j+3} + c_{2j+1} e_{4j+4},\quad
      \mathbf{f^*_{4j+2}} = e_{4j+2},\\
      &\mathbf{f_{4j+3}} = e_{4j+3} + b_{2j+1} e_{4j+4},\quad
      \mathbf{f^*_{4j+3}} = -a_{2j+1} e_{4j+2} + e_{4j+3},
    \end{aligned}
  \end{equation*}
    where the real coefficients $a_n$, $b_n$, $c_n$, $d_n$ are equal to zero whenever $n < 0$, and satisfy the equality
      $c_n + d_n = a_n b_n$ for any $n \geq 0$.
  \begin{prop}
    The given system is an $M$-basis.
  \end{prop}
  \begin{proof}
    The equality $c_n + d_n = a_n b_n$ guarantees the bi\-orthogonality,
      while the completeness of $\fsys$ and $\fstarsys$ is easy to check.
  \end{proof}

  We prove a theorem similar to Theorem~\ref{thm:katavolos}, though we do not investigate the case $k = 1$ in this section.
    \begin{theorem}
    \label{thm:5diag}
      The following statements are equivalent:
      \begin{enumerate}
        \item the given system is rank one dense,
        \item the given system is $k$ point dense for some (equivalently any) $k > 1$,
        \item the sequence
          \[
            \mu_n = \min\left(\frac{1}{\lvert a_n \rvert} + \frac{1}{\lvert b_n \rvert}, \frac{1 + \lvert b_n\rvert}{\lvert d_n\rvert},
                    \frac{1 + \lvert a_n\rvert}{\lvert c_n\rvert}\right)
          \]
        does not belong to $\ell^1$.
      \end{enumerate}
    \end{theorem}
    \begin{proof}
      In order to investigate the density properties we repeat the reasoning from Section~\ref{sec:preliminaries}.
      Presume that $\Xi_n$ are defined by~\eqref{eq:xi}.

      Thus, for any $j \geq 0$ we have
      \begin{equation}
        \label{eq:xi5}
        \begin{aligned}
          \Xi_{4j} &= a_{2j} T_{4j+1, 4j} + c_{2j} T_{4j+2, 4j},\\
          \Xi_{4j + 1} &= -d_{2j} T_{4j+2, 4j} + b_{2j} T_{4j+2, 4j+1},\\
          \Xi_{4j + 2} &= a_{2j+1} T_{4j+2, 4j+3} + c_{2j+1} T_{4j+2, 4j+4},\\
          \Xi_{4j + 3} &= -d_{2j+1} T_{4j+2, 4j+4} + b_{2j+1} T_{4j+3, 4j+4},
        \end{aligned}
      \end{equation}
        where $T_{ij}$ stands for $\langle Te_j, e_i \rangle$.

      First of all we investigate the conditions of rank one density property for $\fsys$.
      \begin{prop}
        \label{prop:inf-dim}
        The following statements are equivalent:
        \begin{enumerate}
          \item the system $\fsys$ is not rank one dense,
          \item there exists an operator $T$ such that $Tr\,T = -1$ and for any $n \geq 0$ one has \[\Xi_n + \sum_{m=0}^n \inner{Te_m}{e_m} = 0,\]
          \item the sequence $\left\{\mu_n\right\}_{n=1}^\infty$ belongs to $\ell^1$.
        \end{enumerate}
      \end{prop}
      \begin{proof}
        The equivalence of the first two statements is due to Proposition~\ref{prop:reformulation}.
        We are going to prove the equivalence between the last two statements.

        Assume that $\left\{\mu_n\right\}_{n=1}^\infty \in \ell^1$; our purpose is to construct the required operator $T$.
        Let $T_{00}$ be equal to $-1$, and $T_{jj}$ be equal to zero for any $j > 0$.
        Next we consider three cases for each $n \geq 0$.

        \noindent\textbf{Case 1.}
        Suppose $\mu_n = 1/\lvert a_n\rvert + 1/\lvert b_n \rvert$.
        For $n=2j$ we set
        \[
          T_{4j+1,4j}=1/a_n, \quad T_{4j+2,4j} = 0, \quad T_{4j+2,4j+1}=1/b_n.
        \]
        That guarantees the equality $\Xi_{2n} = \Xi_{2n+1} = 1$.
        For $n=2j+1$ we set
        \[
          T_{4j+2,4j+3}=1/a_n, \quad T_{4j+2,4j+4} = 0, \quad T_{4j+3,4j+4}=1/b_n,
        \]
        which provides the equality $\Xi_{2n} = \Xi_{2n+1} = 1$.

        \medskip
        \noindent\textbf{Case 2.}
        Assume $\mu_n = (1 + \lvert b_n\rvert)/\lvert d_n\rvert$.
        For $n=2j$ we set
        \[
          T_{4j+1,4j} = b_{2j}/d_{2j}, \quad T_{4j+2,4j} = -1/d_{2j}, \quad T_{4j+2,4j+1} = 0.
        \]
        Again, we have $\Xi_{2n} = \Xi_{2n+1} = 1$.
        For $n = 2j + 1$ we set
        \[
          T_{4j+2,4j+3}=b_{2j+1}/d_{2j+1},  \quad T_{4j+2,4j+4} = -1/d_{2j+1}, \quad T_{4j+3,4j+4}=0,
        \]

        The third case $\mu_n = (1 + \lvert a_n\rvert)/\lvert c_n\rvert$ is left to the reader.
        \medskip

        All the other entries $T_{ij}$ we set to zero.
        These equalities ensure that $\Xi_n = -\sum_{s=0}^n T_{ss} = 1$ for any $n \geq 0$.

        The constructed operator $T$ belongs to the trace class since the non-zero operator matrix entries are summable
          due to the assumption that $\left\{\mu_n\right\}_{n=1}^\infty \in \ell^1$.
        Since the trace of $T$ is equal to $-1$, the sufficiency is proved.

        \medskip
        Conversely, assume that there exists a trace class operator $T$ in the annihilator of $R_1(\cal{A})$ with the trace equal to $-1$.

        First of all, we prove that $\{\mu_{2j}\}_{j=1}^\infty$ is a summable sequence.
        Since $T$ is in the trace class, the sequence of vectors $\nu_n = \lvert T_{nn} \rvert + \lvert T_{n, n + 1} \rvert + \lvert T_{n, n+2} \rvert$
          belongs to $\ell^1$.
        Obviously, $\{\Xi_n\}_{n=1}^\infty$ belongs $\ell^1$ as well.
        As a consequence, we have $\lvert \Xi_n\rvert \geq 0.5$ for all $n$ large enough; we will assume that it holds for any $n > 0$.

        It can be easily checked that if for some $n$ one of the numbers $a_n,b_n,c_n,d_n$ is equal to zero, then $\nu_n \geq \mu_n/2$.
        From this point we will suppose that the coefficients are nonzero for any $n > 0$.

        For any even $n = 2j$ consider the linear function
        \begin{align*}
          g_{n}(x) = \Big\lvert \frac{\Xi_{2n} - c_{n} x}{a_{n}} \Big\rvert +
                       \lvert x \rvert +
                     \Big\lvert \frac{\Xi_{2n+1} + d_{n} x}{b_{n}} \Big\rvert.
        \end{align*}
        Obviously, we have $\nu_{n} = g_{n}(T_{2n+2, 2n})$.
        The function $g_n$ is piecewise linear, so its minimum is attained in the breakpoints.
        The breakpoints are zero, $y_n = \Xi_{2n}/c_n$ and $z_n = -\Xi_{2n+1}/d_n$.
        We have $g_n(0) \geq \mu_n/2$.
        Consider the set $N_1 \subseteq \mathbb{N}_{{even}}$, such that for any $n \in N_1$ the function $g_n$ attains its minimum in
          the point $\Xi_{2n}/c_n$.
        Thus, for any $n\in N_1$ we have $\nu_n \geq g_n(y_n)$.
        We have
        \[
          g_n(y_n) = \Big\lvert \Xi_{2n}/c_n \Big\rvert +
                   \Big\lvert \frac{\Xi_{2n+1} + d_{n} \Xi_{2n}/c_n}{b_{n}} \Big\rvert.
        \]
        Since $\nu_n$ is summable and $\nu_n \geq g_n(y_n) \geq 0.5/|c_n|$ for any $n \in N_1$
          we deduce that $\sum_{n \in N_1} |c_n|^{-1} < \infty$.

        Clearly, $G_n = g_n(y_n) - \Big\lvert \Xi_{2n}/c_n \Big\rvert$ is summable.
        Let $\Delta_n$ stand for the difference $\left(\Xi_{2n+1} - \Xi_{2n}\right)$.
        Then
        \[
          G_n = \Big\lvert \frac{c_n \Xi_{2n+1} + d_{n} \Xi_{2n}}{c_n b_n} \Big\rvert
              = \Big\lvert \frac{c_n \Delta_n + (c_n + d_n) \Xi_{2n}}{c_n b_n} \Big\rvert
              = \Big\lvert \frac{c_n \Delta_n + a_n b_n\Xi_{2n}}{c_n b_n} \Big\rvert.
        \]
        Hence, $\big\lvert G_n - \Xi_{2n}\lvert a_n/c_n\rvert \big\rvert \leq \lvert \Delta_n/b_n \rvert$.

        Consider the sets $N_2 = \{n\in N_1 \mid 0.5 \leq \lvert b_n \rvert\}$ and $N_3 = N_1 \setminus N_2$.
        Since $\Xi_n$ has a finite limit, we have $\sum_{n \in N_2} \lvert a_n/c_n \rvert < \infty$.
        Hence, $\{\mu_n\}_{n \in N_2} \in \ell^1$.

        Assume that $\sum_{n \in N_3} \lvert a_n/c_n\rvert = \infty$.

        We have $\big\lvert \lvert b_n \rvert G_n - \Xi_{2n}\lvert (c_n + d_n)/c_n \rvert\big\rvert \leq \lvert \Delta_n \rvert$.
        Since the sequences $\{b_n G_n\}_{n\in N_3}$ and $\{\Delta_n\}_{n \in N_3}$ are absolutely summable, the sequence
          $\{(c_n + d_n) / c_n\}_{n\in N_3} $ is absolutely summable as well.
        Consequently, $\lvert d_n/c_n \rvert \geq 0.5$ when $n$ is large enough.
        We get $\{1/c_n\}_{n \in N_3} \in \ell^1$, and thus $\{1/d_n\}_{n \in N_3} \in \ell^1$.
        Since for $n \in N_3$ one has $\lvert b_n \rvert \leq 0.5$, we have
        \[
          \sum_{n\in N_3} \mu_n \leq \sum_{n \in N_3} \frac{1 + \lvert b_n\rvert}{\lvert d_n \rvert} < \infty.
        \]

        Repeating the reasoning for odd $n$, we get that $\left\{\mu_n\right\}_{n=1}^\infty$ is a summable sequence.
      \end{proof}
      \bigskip
      Now consider the case of the $k$-dimensional operator $T = \sum_{s=1}^k y^s \otimes x^s$, where $x^s, y^s \in \cal{H}$.
      This time we define the vectors $\seq{v}$ and $\seq{u}$ in $\mathbb{R}^k$ as follows:
      \begin{align*}
        v_{2j} &= (y^1_{4j}, y^2_{4j}, \dots ,y^k_{4j}) \quad
        &v^*_{2j} = (x^1_{4j}, x^2_{4j}, \dots ,x^k_{4j}) \\
        v_{2j+1} &= (x^1_{4j+2}, x^2_{4j+2}, \dots ,x^k_{4j+2}) \quad
        &v^*_{2j+1} = (y^1_{4j+2}, y^2_{4j+2}, \dots ,y^k_{4j+2}) \\
        u_{2j} &= (x^1_{4j+1}, x^2_{4j+1}, \dots ,x^k_{4j+1}) \quad
        &u^*_{2j} = (y^1_{4j+1}, y^2_{4j+1}, \dots ,y^k_{4j+1}) \\
        u_{2j+1} &= (y^1_{4j+3}, y^2_{4j+3}, \dots ,y^k_{4j+3}) \quad
        &u^*_{2j+1} = (x^1_{4j+3}, x^2_{4j+3}, \dots ,x^k_{4j+3})
      \end{align*}
      Note that the sequences $\seq{v}$, $\seq{u}$ belong to $\ell^2(\mathbb{R}^k)$.

      Now we can rewrite the equations~\eqref{eq:xi5} using the introduced vectors:
      \begin{align*}
        \Xi_{4j} &= a_{2j} \langle u_{2j}, v_{2j}\rangle + c_{2j} \langle v_{2j+1}, v_{2j}\rangle,\\
        \Xi_{4j + 1} &= -d_{2j} \langle v_{2j+1}, v_{2j}\rangle + b_{2j} \langle v_{2j+1}, u_{2j}\rangle,\\
        \Xi_{4j + 2} &= a_{2j+1} \langle v_{2j+1}, u_{2j+1} \rangle + c_{2j+1} \langle v_{2j+1}, v_{2j+2} \rangle,\\
        \Xi_{4j + 3} &= -d_{2j+1} \langle v_{2j+1}, v_{2j+2}\rangle + b_{2j+1} \langle u_{2j+1}, v_{2j+2} \rangle,
      \end{align*}
        where $\langle\cdot, \cdot\rangle$ denotes the scalar product in $\mathbb{R}^k$.
      These equations simplify to
      \begin{equation}
        \label{eq:vector-eqs}
        \begin{aligned}
          \Xi_{2j} &= a_{j} \langle u_{j}, v_{j} \rangle  + c_{j} \langle v_{j+1}, v_{j} \rangle,\\
          \Xi_{2j + 1} &= -d_{j} \langle v_{j+1}, v_{j} \rangle + b_{j} \langle v_{j+1}, u_{j}\rangle.
        \end{aligned}
      \end{equation}

      Next we are going to analyze the necessary condition of $k$ point density property for $\fsys$.
      \begin{prop}
        \label{prop:2pd}
        If $\left\{\mu_n\right\}_{n=1}^\infty$ belongs to $\ell^1$ then it is possible to construct
          the vector sequences $\seq{u}, \seq{v} \in \ell^2(\mathbb{R}^2)$ such that for any $n \geq 0$ we have $\Xi_n = 1$.
      \end{prop}
      \begin{corol}
        \label{corol:2density}
        If $\fsys$ is $k$ point dense for any $k \geq 2$ then $\sum_{n=1}^\infty \lvert\mu_n\rvert = \infty$.
      \end{corol}
      \begin{proof}[Proof of the corollary]
        Assume the converse: $\{\mu_n\}_{n=1}^\infty \in \ell^1$.

        We apply the proposition and get the vectors $u_n$ and $v_n$.
        Now without loss of generality we can assume that $u_0 \neq 0$.
        Then consider $u^*_0$ so that $\inner{u_0}{u^*_0} = -1$ and set all $u^*_n$ ($n > 0$), $v^*_n$ to zero.
        Since the trace of the resulting operator $T$ is equal to $\sum_{n=0}^\infty \left(\inner{u_n}{u^*_n} + \inner{v_n}{v^*_n}\right) \neq 0$,
          Proposition~\ref{prop:kreformulation} implies that $\fsys$ is not two point dense.
        Trivially, when $\fsys$ is not two point dense, it is also not $k$ point dense for any $k \geq 2$.
      \end{proof}
      \begin{proof}[Proof of Proposition~\ref{prop:2pd}]
        First we are going to present the vector lengths $V_n = \lvert v_n \rvert$ for each $n \geq 0$.

        For this purpose we are going to define an auxiliary sequence $\{M_n\}_{n=1}^\infty \in \ell^2$ such that $V_n \geq M_n$ for any $n$.
        On each step $n$ we will define $V_n$ and $M_{n+1}$.

        We start by setting $V_0 = M_1 = 1$.

        For any $n > 0$ we have three choices for $\mu_n$:
        \begin{enumerate}
          \item if $\mu_n = 1/\lvert a_n \rvert+ 1/\lvert b_n \rvert$, we set 
              \[
                M_{n+1} = \frac{1}{\sqrt{\smash[b]{\lvert b_n \rvert}}}, \quad
                V_n = \max\left(M_n, \frac{1}{\sqrt{\smash[b]{\lvert a_n \rvert}}}\right),
              \]
          \item whenever $\mu_n = (1 + \lvert a_n \rvert)/\lvert c_n \rvert$, we set 
              \[
                M_{n+1} = \max\biggl(\smash[b]{\frac{\sqrt{|a_n|}}{\sqrt{\lvert c_n \rvert}}}, \frac{1}{\sqrt{\smash[b]{\lvert c_n \rvert}}}\biggr), \quad
                V_n = \max\left(M_n, \frac{2}{\sqrt{\smash[b]{|c_n|}}}\right),
              \]
          \item if $\mu_n = (1 + \lvert b_n \rvert)/\lvert d_n \rvert$, we set 
              \[
                M_{n+1} = \max\biggl(\smash[t]{\frac{\sqrt{\lvert b_n \rvert}}{\sqrt{\lvert d_n \rvert}}}, \frac{1}{\sqrt{\smash[b]{\lvert d_n \rvert}}}\biggr), \quad
                V_n = \max\left(M_n, \frac{2 + \sqrt{\lvert b_n\rvert}}{\sqrt{\smash[b]{\lvert d_n \rvert}}}\right).
              \]
        \end{enumerate}
        Now all $V_n$, $M_n$ are set, and obviously $V_n \geq M_n$ for any $n > 0$.
        
        Next we are going to present the vector lengths $U_n = \lvert u_n \rvert$, for each $n \geq 0$.
        Set $U_0$ to zero and for any $n > 0$ we have three cases again:
        \begin{enumerate}
          \item when $\mu_n = 1/\lvert a_n \rvert+ 1/\lvert b_n \rvert$, we set 
              $
                U_n = \displaystyle\sqrt{\frac{1}{a_n^2 V_n^2} + \frac{1}{b_n^2 V_{n+1}^2}},
              $ 
          \item whenever $\mu_n = (1 + \lvert a_n \rvert)/\lvert c_n \rvert$, we set 
              $
                U_n =  \dfrac{a_n V_n}{\displaystyle\sqrt{c_n^2 V_{n+1}^2 V_n^2 - 1}},
              $
          \item if $\mu_n = (1 + \lvert b_n \rvert)/\lvert d_n \rvert$, we set 
              $
                U_n = \dfrac{b_n V_{n+1}}{\displaystyle\sqrt{ d_n^2 V_{n+1}^2 V_n^2 - 1}}.
              $
        \end{enumerate}

        Due to our choice of $V_n$, $M_n$ the values $U_n$ are well-defined for each $n > 0$.

        \begin{lemma}
          \label{lemma}
          If for some nonzero real $A$, $B$, $X$, $Y$, $Z$ we have $(AX)^{-2} + (BY)^{-2} = Z^2$, there exist such vectors $x$, $y$, $z$ with lengths
            $X$, $Y$, $Z$ correspondingly such that
          \begin{equation}
            \label{eqn:system}
            \begin{aligned}
              \langle x, y \rangle &= 0,\\
              \langle x, z \rangle &= 1/A,\\
              \langle y, z \rangle &= 1/B.
            \end{aligned}
          \end{equation}
        \end{lemma}
        \begin{proof}
          Take $\alpha$ such that $\cos \alpha  = 1/(AXZ)$ and $\sin \alpha  = 1/(BYZ)$.
          Consider three vectors $x$, $y$, $z$ in $\mathbb{R}^2$ with lengths $X$, $Y$, $Z$ such that
            $\angle(y, z) = \pi/2 - \alpha$ and $\angle(x, z) = \alpha$.
          Clearly, the vectors $x$ and $y$ must be orthogonal now.
          The equations~\eqref{eqn:system} are trivial to check.
        \end{proof}
        \begin{prop}
          For any $n\geq 0$ there are vectors $u_n$, $v_n$ with lengths $U_n$, $V_n$ in $\mathbb{R}^2$ such that~\eqref{eq:vector-eqs} are satisfied.
        \end{prop}
        \begin{proof}
          We argue by induction.
          We start with $v_0 = (0, 1)$ and $u_0 = 0$.
          Suppose that we have constructed a sequence of vectors $v_m, u_m$ for all $m < n$ and $v_n$.
          We are going to build $u_n$ and $v_{n+1}$.
          We consider three cases for $\mu_n$.

          In the first case the chosen $U_n$, $1/\lvert a_n V_n \rvert$ and $1/\lvert b_n  V_{n+1}\rvert$
              form a right triangle with hypotenuse $U_n$, and so here Lemma~\ref{lemma} can be applied.
          It follows that there are vectors $u'_n$, $v'_n$, $v'_{n+1}$ in $\mathbb{R}^2$ with lengths $U_n$, $V_n$, $V_{n+1}$ correspondingly such that
          \begin{equation}
            \label{eq:system1}
            \begin{aligned}
              \langle u'_n, v'_n \rangle &= 1/a_n,\\
              \langle u'_n, v'_{n+1} \rangle &= 1/b_n,\\
              \langle v'_n, v'_{n+1} \rangle &= 0,
            \end{aligned}
          \end{equation}
          which in turn yields the equations~\eqref{eq:vector-eqs}.
          Now we can simply rotate the triple $(u'_n, v'_n, v'_{n+1})$ so that $v'_n$ coincides with $v_n$.
          We will set $u_n$ and $v_{n+1}$ to the rotated $u'_n$ and $v'_{n+1}$ accordingly.
          Since the rotation preserves the scalar product inside the triple, the equations~\eqref{eq:system1} hold for $u_n$, $v_n$, $v_{n+1}$ as well.

          In the second case the chosen $V_{n+1}$, $a_n/(c_n U_n)$ and $1/(c_n V_n)$ also form a right triangle and Lemma~\ref{lemma} applies here
            as well.
          It implies that there are vectors $u'_n$, $v'_n$, $v'_{n+1}$ in $\mathbb{R}^2$ with lengths $U_n$, $V_n$, $V_{n+1}$ correspondingly such that
          \begin{equation}
            \begin{aligned}
              \langle u'_n, v'_n \rangle &= 0,\\
              \langle u'_n, v'_{n+1} \rangle &= a_n/c_n,\\
              \langle v'_n, v'_{n+1} \rangle &= 1/c_n,
            \end{aligned}
          \end{equation}
          and the equations~\eqref{eq:vector-eqs} follow from that.
          Using rotation again, we receive $u_n$, $v_n$, $v_{n+1}$.

          In the third case the chosen $V_n$, $b_n/(d_n U_n)$ and $1/(d_n V_{n+1})$ also form a right triangle and Lemma~\ref{lemma} applies here
            as well.
          It implies that there are vectors $u'_n$, $v'_n$, $v'_{n+1}$ in $\mathbb{R}^2$ with lengths $U_n$, $V_n$, $V_{n+1}$ correspondingly such that
          \begin{equation}
            \begin{aligned}
              \langle u'_n, v'_n \rangle &= b_n/d_n,\\
              \langle u'_n, v'_{n+1} \rangle &= 0,\\
              \langle v'_n, v'_{n+1} \rangle &= -1/d_n,
            \end{aligned}
          \end{equation}
          and the equations~\eqref{eq:vector-eqs} are also true.
          One more time we do the rotation, and we get $u_n$, $v_n$, $v_{n+1}$.
        \end{proof}
        \begin{prop}
          The following inequalities are true.
          \begin{align*}
            M_{n+1} &\leq \sqrt{\mu_n},\\
            V_n &\leq \max(2\sqrt{\mu_n}, M_n),\\
            U_n &\leq 2\sqrt{\mu_n}.
          \end{align*}
        \end{prop}
        \begin{proof}
          First two statements are trivial.
          For the last statement we consider the same three cases.

          In the first case we have $\lvert a_n V_n\rvert \geq \sqrt{\lvert a_n \rvert}$ and
            $\lvert b_n V_{n+1} \rvert \geq  \lvert b_n M_{n+1}\rvert = \sqrt{\lvert b_n \rvert}$.
          It follows that $U_n \leq \sqrt{\mu_n}$.

          In the second case we have $c_n V_{n+1} V_n \geq c_n M_{n+1} V_n \geq 2$ and so
          \[
            U_n \leq \dfrac{\lvert a_n\rvert V_n}{\sqrt{\frac{3}{4} c_n^2 V_{n+1}^2 V_n^2}} \leq 2\dfrac{\lvert a_n\rvert}{\lvert c_n\rvert V_{n+1}}
                \leq 2\dfrac{\lvert a_n\rvert}{\lvert c_n\rvert M_{n+1}} \leq 2\sqrt{\dfrac{\lvert a_n\rvert}{\lvert c_n\rvert}} < 2\sqrt{\mu_n}.
          \]

          In the third case we have $d_n V_{n+1} V_n \geq 2$ and hence
          \[
            U_n \leq \dfrac{\lvert b_n\rvert V_{n+1}}{\sqrt{\frac{3}{4} d_n^2 V_{n+1}^2 V_n^2}} \leq 2\dfrac{\lvert b_n\rvert}{\lvert d_n\rvert V_{n}}
                \leq 2\sqrt{\dfrac{\lvert b_n\rvert}{\lvert d_n\rvert}} < 2\sqrt{\mu_n}.
          \]
        \end{proof}

        The last proposition implies that $V_n$ is bounded up to some constant by $\max(\sqrt{\mu_{n-1}}, \sqrt{\mu_n})$, and
          $U_n \leq \sqrt{\mu_n}$ for any $n > 0$.
        Hence, the constructed sequences $V_n$ and $U_n$ belong to $\ell^2$.
        That finishes the proof of Proposition~\ref{prop:2pd}.
      \end{proof}
      Now due to Corollary~\ref{corol:2density} we get that two point density of $\fsys$ implies the divergence of $\sum_{n=1}^\infty \lvert \mu_n\rvert$,
        which in turn is equivalent to rank one density property of $\fsys$ (see Proposition~\ref{prop:inf-dim}).
      Since rank one density implies $k$ point density for any $k$, Theorem~\ref{thm:5diag} is proved.
    \end{proof}

\bigskip

\section{Acknowledgements}
  The author gratefully acknowledges the many helpful suggestions of Anton Baranov during the preparation of the paper.

\end{document}